\documentclass[12pt]{amsart}

\usepackage{amssymb}
\usepackage{graphicx}
\usepackage{hyperref}
\usepackage{ifthen}
\usepackage{xargs}
\usepackage{xspace}

\newcommand{\CH}{\mathtt{CH}}
\newcommand{\DC}{\mathtt{DC}}
\newcommand{\ZF}{\mathtt{ZF}}
\newcommand{\ZFC}{\mathtt{ZFC}}

\newcommand{\definedterm}[1]{\emph{#1}}

\newcommand{\Baire}{Baire\xspace}
\newcommand{\Borel}{Bor\-el\xspace}
\newcommand{\Cantor}{Can\-tor\xspace}
\newcommand{\Cohen}{Co\-hen\xspace}
\newcommand{\Galvin}{Gal\-vin\xspace}
\newcommand{\Polish}{Po\-lish\xspace}

\newcommand{\Cantorspace}{2^\N}
\newcommand{\Cantortree}{2^{<\N}}
\newcommand{\clique}[2]{\mathrm{CLQ}^{\mathord{\ge} #1}_{#2}}
\newcommand{\closedinterval}[2]{[#1, #2]}
\newcommand{\closedopeninterval}[2]{[#1, #2)}
\newcommand{\independent}[2]{\mathrm{IND}^{#1}_{#2}}
\newcommand{\LiYorke}[1]{\mathrm{SCR}_{#1}}

\newcommand{\distributionallyscrambled}[1]{\mathrm{DSCR}_{#1}}
\newcommand{\distributionallyseparated}[1]{\mathrm{DSEP}_{#1}}
\newcommand{\N}{\mathbb{N}}
\newcommand{\positiveintegers}{\N \setminus \set{0}}
\newcommand{\saturation}[2]{[#1]_{#2}}
\newcommand{\separated}[1]{\mathrm{SEP}_{#1}}

\newcommand{\Z}{\mathbb{Z}}

\newcommand{\absolutevalue}[1]{|#1|}
\newcommand{\bernoullishift}{S}
\newcommand{\calN}{\mathcal{N}}
\newcommand{\cardinality}[1]{|#1|}
\newcommand{\continuum}{\mathfrak{c}}
\newcommand{\composition}{\circ}
\newcommand{\constantsequence}[2]{#1^{#2}}
\newcommand{\extensions}[1]{\calN_{#1}}
\newcommand{\from}{\colon}
\newcommand{\Fsigma}{F_\sigma}
\newcommandx{\functions}[3][3 =]{
  \ifthenelse{\equal{#3}{}}{{#2}^{#1}}{{#2}^{#1}_{#3}}
}
\newcommand{\goesto}{\to}
\newcommand{\image}[2]{#1(#2)}
\newcommand{\injective}[2]{\mathrm{INJ}^{#1}_{#2}}
\newcommandx{\intersection}[2][1 =, 2 =]{
  \ifthenelse{\equal{#1}{}}{\cap}{
    \ifthenelse{\equal{#2}{}}{\bigcap #1}{{\bigcap_{#1} #2}}
  }
}
\newcommand{\inverse}[1]{#1^{-1}}
\newcommand{\length}[1]{|#1|}

\newcommand{\lowermetric}[1]{\underline{d}_{#1}}

\newcommand{\lowerproductmetric}{\underline{\rho}}

\newcommand{\mathcomma}{\text{, }}
\newcommand{\mathcommaand}{\text{, and }}
\newcommand{\metric}[1]{d_{#1}}
\newcommand{\pair}[2]{(#1, #2)}
\newcommand{\preimage}[2]{\inverse{#1}(#2)}
\newcommandx{\product}[2][1 =, 2 =]{
  \ifthenelse{\equal{#1}{}}{\times}{
    \ifthenelse{\equal{#2}{}}{\prod #1}{{\prod_{#1} #2}}
  }
}
\newcommand{\productmetric}{\rho}
\newcommand{\projection}[1]{\mathrm{proj}_{#1}}
\newcommand{\quadruple}[4]{(#1, #2, #3, #4)}
\renewcommand{\restriction}[2]{#1 \upharpoonright #2}
\newcommandx{\sequence}[2][2 = undefined]{
  \ifthenelse{\equal{#2}{undefined}}{(#1)}{
    (#1)_{#2}
  }
}
\newcommandx{\set}[2][2 = undefined]{
  \ifthenelse{\equal{#2}{undefined}}{\{ #1 \}}{
    \{ #1 \suchthat #2 \}
  }
}
\newcommand{\setcomplement}[1]{\twiddle #1}
\newcommand{\suchthat}{\mid}
\newcommand{\support}{\mathrm{supp}}

\newcommand{\twiddle}{\raisebox{1.5pt}{\scalebox{.75}{$\mathord{\sim}$}}}
\newcommandx{\union}[3][1 =, 2 =, 3 =]{
  \ifthenelse{\equal{#1}{}}{\cup}{
    \ifthenelse{\equal{#2}{}}{\bigcup #1}{
      \ifthenelse{\equal{#3}{}}{\bigcup_{#1} #2}{\bigcup_{#1}^{#2}{#3}}
    }
  }
}

\newcommand{\uppermetric}[1]{\overline{d}_{#1}}
\newcommand{\upperproductmetric}{\overline{\rho}}

\newtheorem{introtheorem}{Theorem}
\newtheorem{lemma}{Lemma}[section]
\newtheorem{proposition}[lemma]{Proposition}
\newtheorem{theorem}[lemma]{Theorem}

\theoremstyle{definition}
\newtheorem*{acknowledgements}{Acknowledgements}
\newtheorem*{remark}{Remark}

\begin{document}

\begin{abstract}
  For all $n \ge 2$, we obtain a homeomorphism of \Cantor space that has a
  distributionally $n$-scrambled \Cantor set containing a distributionally
  $(n+1)$-scrambled uncountable set, but does not have an
  $(n+1)$-scrambled \Cantor set.
\end{abstract}

\author{B. Miller}
\address{
  B. Miller \\
  1008 Balsawood Drive \\
  Durham, NC 27705
 }
\email{glimmeffros@gmail.com}
\urladdr{\url{https://glimmeffros.github.io}}

\keywords{distributional chaos, Li--Yorke chaos, scrambled set}

\subjclass{Primary 37B05; Secondary 03E15}

\title{$n$-Scrambled Cantor Sets}

\maketitle

\section*{Introduction}

Given $n \ge 2$, a metric space $X$, $f \from X \to X$, and $x \in \functions
{n}{X}$, define $\lowermetric{X}(x) = \min_{j < k < n} \metric{X}(x(j),x(k))$
and $\uppermetric{X}(x) = \max_{j < k < n} \metric{X}(x(j), x(k))$. We say that
$x$ is \definedterm{proximal} if $\liminf_{i \goesto \infty} \uppermetric{X}(f^i
\composition x) = 0$, \definedterm{distributionally proximal} if $\forall \epsilon
> 0 \ \limsup_{i \goesto \infty} \cardinality{\set{i'<i}[\uppermetric{X}(f^{i'}
\composition x)<\epsilon]} / i=1$, \definedterm{separated} if $\limsup_{i
\goesto \infty} \lowermetric{X}(f^i \composition x) > 0$, \definedterm
{distributionally separated} if $\exists \epsilon > 0 \ \liminf_{i \goesto \infty}
\cardinality{\set{i'<i}[\lowermetric{X}(f^{i'} \composition x)<\epsilon]} / i=0$,
and \definedterm{(distributionally) scrambled} if it is (distributionally) proximal
and (distributionally) separated. A set $Y \subseteq X$ is \definedterm
{(distributionally) $n$-scrambled} if every injection in $\functions{n}{Y}$ is
(distributionally) scrambled, and \definedterm{(distributionally) finitely
scrambled} if it is (distributionally) $n$-scrambled for all $n \ge 2$.

In \cite{GeschkeGrebikMiller}, it was shown that $2$-scrambled uncountable
sets yield $2$-scrambled \Cantor sets in \Polish dynamical systems. Here we
show that the analogous statement for larger $n$ is false:

\begin{introtheorem}[$\ZFC$] \label{introtheorem:main}
  Suppose that $n \ge 2$. Then there is a homeomorphism $T \from
  \Cantorspace \to \Cantorspace$ that has a distributionally $n$-scrambled \Cantor
  set containing a distributionally finitely scrambled uncountable set, yet does not have
  an $(n+1)$-scrambled \Cantor set.
\end{introtheorem}

Let $\injective{n}{X}$ denote the set of injective sequences of elements of $X$
of length $n$. Let (D)$\LiYorke{f}$ denote the set of
(distributionally) scrambled finite sequences.
A \definedterm{graph} on $X$ is an irreflexive symmetric set $G \subseteq X
\times X$. We say that $Y$ is a \definedterm{$G$-clique} if its distinct
points are $G$-related, and \definedterm{$G$-independent} if its points are
not $G$-related. Let $\clique{n}{G}$ denote the set of injective finite sequences
of length at least $n$ whose images are $G$-cliques, and let $\independent
{n}{G}$ denote the set of injective sequences of length $n$ whose images are
$G$-independent.
A \definedterm{homomorphism} from a relation $R$ on $X$ to a relation $S$ on $Y$ is a
function $\pi \from X \to Y$ whose composition with each sequence in $R$ is in
$S$. Such a function is a \definedterm{homomorphism} from $\pair{R}{R'}$ to
$\pair{S}{S'}$ if it is also a homomorphism from $R'$ to $S'$. The \definedterm
{saturation} of $Y$ under a bijection $T \from X \to X$ is given by
$\saturation{Y}{T} = \union[i \in \Z][\image{T^i}{Y}]$. A subset of a
topological space is \definedterm{$\Fsigma$} if it is a countable union of
closed sets. The primary new technical result underlying our proof of
Theorem \ref{introtheorem:main} is as follows:

\begin{introtheorem}[$\ZF + \DC$] \label{introtheorem:technical}
  Suppose that $G$ is an $\Fsigma$ graph on $\Cantorspace$ and $n \ge 2$.
  Then there exist a homeomorphism $T \from \Cantorspace \to \Cantorspace$
  and a continuous injective homomorphism $\pi \from \Cantorspace \to
  \Cantorspace$ from $\pair{\injective{n}{\Cantorspace} \union \clique{n+1}{G}}
  {\independent{n+1}{G}}$ to $\pair{\distributionallyscrambled{T}}{\setcomplement{\LiYorke
  {T}}}$ for which $\saturation{\image{\pi}
  {\Cantorspace}}{T}$ is co-countable.
\end{introtheorem}

In \S\ref{section:coding}, we establish Theorem \ref{introtheorem:technical}.
In \S\ref{section:consequences}, we obtain Theorem \ref{introtheorem:main}
and a related independence phenomenon as corollaries.

\section{Hypergraph homomorphisms} \label{section:coding}

Define $\bernoullishift \from \functions{\Z}{2} \to \functions{\Z}{2}$ by
$\bernoullishift(c)(i) = c(i+1)$ for all $i \in \Z$ and $c \in \functions{\Z}{2}$.
Let (D)$\separated{f}$ denote the set of (distributionally) separated
finite sequences. As (distributional) scrambling is preserved by topological conjugacy of
compact metric spaces, Theorem \ref{introtheorem:technical} is a
consequence of the topological characterization of $\Cantorspace$
(see, for example, \cite[Proposition 7.4]{Kechris}) and:

\begin{theorem}[$\ZF + \DC$] \label{theorem:technical}
Suppose that $G$ is an $\Fsigma$ graph on $\Cantorspace$ and $n \ge 2$.
Then there exist an $\bernoullishift$-invariant perfect set $F \subseteq
\functions{\Z}{2}$ and a continuous injective homomorphism $\pi \from
\Cantorspace \to F$ from $\pair{\injective{n}{\Cantorspace} \union \clique
{n+1}{G}}{\independent{n+1}{G}}$ to $\pair{\distributionallyscrambled{\bernoullishift}}
{\setcomplement{\separated{\bernoullishift}}}$ for which $F \setminus
\saturation{\image{\pi}{\Cantorspace}}{\bernoullishift}$ is countable.
\end{theorem}

\begin{proof}
For all $s \in \Cantortree$, set $\extensions{s} = \set{c \in \Cantorspace}
[\forall i < \absolutevalue{s} \ s(i) = c(i)]$. Given $\ell \in \N$ and $n' \ge n$,
we say that a function $p \from \functions{\ell}{2} \to n'$ is a \definedterm
{near bijection} if it induces a bijection of $\functions{\ell}{2} \setminus
\preimage{p}{\set{0}}$ with $n' \setminus \set{0}$ but $\set{\constantsequence
{0}{\ell}} \subsetneq \preimage{p}{\set{0}}$. For all $0 < k < n'$, let $\inverse{p}
(k)$ denote the unique element of $\preimage{p}{\set{k}}$.

The construction uses three types of coding blocks. The first
ensures injectivity, the second separates arbitrary $n$-tuples,
and the third separates larger $G$-cliques. In the third type,
spacing the potentially non-zero coordinates by $m+1$ ensures that
separation of an $(n+1)$-tuple at a fixed scale bounds $m$.
Long repetitions yield distributional separation, while long
intervening gaps yield distributional proximality.

Fix an increasing sequence $\sequence{G_m}[m \in \N]$ of closed graphs on $\Cantorspace$ whose
union is $G$. Define $G_{\ell,m} = \set
{\pair{c}{c'} \in \Cantorspace \times \Cantorspace}[G_m \intersection
(\extensions{\restriction{c}{\ell}} \times
\extensions{\restriction{c'}{\ell}}) \neq \emptyset]$. Let $P_\ell$ denote the set of near bijections
$p \from \functions{\ell}{2} \to n$, and
let $P_{\ell,m,n'}$ denote the set of
near bijections $p \from \functions{\ell}{2} \to n'$ such that
$\extensions{\inverse{p}(j)} \times \extensions{\inverse{p}(k)}
\subseteq G_{\ell,m}$ for all $0 < j < k < n'$. Fix an infinite set
$I \subseteq \positiveintegers$, a bijection
$\phi \from I \to \N \union \union[\ell \in \N][(\set{\ell} \times P_\ell)]
\union \union[\ell,m \in \N, n' > n][(\set{\ell} \times \set{m}
\times \set{n'} \times P_{\ell,m,n'})]$, and a function
$r \from I \to \positiveintegers$ with the property that $i^2 \le r(i)$ and
$I \intersection \closedinterval{i}{i(i+r(i)q(i))} = \set{i}$ for
all $i \in I$, where $w(i)=0$ if $\phi(i) \in \N$, $w(i)=n-2$ if
$\phi(i) \in \set{\ell} \times P_\ell$ for some $\ell \in \N$,
$w(i)=(m+1)(n'-2)$ if
$\phi(i) \in \set{\ell} \times \set{m} \times \set{n'} \times P_{\ell,m,n'}$ for some
$\ell,m \in \N$ and
$n' > n$, and $q(i) = \max(1, 2w(i))$. Define a continuous function
$\pi \from \Cantorspace \to \functions{\Z}{2}$ by setting
$\pi(c)(i') = 1$ if and only if either $i' = 0$ or there exist $i \in I$ and $t < r(i)$
such that one of the following holds:

\begin{enumerate}
  \item $\phi(i) \in \N$, $i' = i + tq(i)$, and $c(\phi(i)) = 1$;
  \item $\exists 0 < k < n \exists \ell \in \N \exists p \in
    P_\ell$ \\
      \hspace*{6pt}
      $\pair{\ell}{p} = \phi(i) \mathcomma i' = i + tq(i)+k-1 \mathcommaand p
        (\restriction{c}{\ell}) = k$; or
  \item $\exists n' > n \exists 0 < k < n' \exists \ell,m \in \N \exists p \in
    P_{\ell,m,n'}$ \\
    \hspace*{6pt} $\quadruple{\ell}{m}{n'}{p}=\phi(i) \mathcomma i' = i + tq(i) +
      (k-1)(m+1) \mathcommaand$ \\
    \hspace*{12pt} $p(\restriction{c}{\ell}) = k$.
\end{enumerate}

\begin{lemma} \label{lemma:injective}
$\pi$ is injective.
\end{lemma}

\begin{proof}
Suppose that $c, d \in \Cantorspace$ are distinct. Then there exists $k \in
\N$ with $c(k) \neq d(k)$. Fix $i \in I$ for which $\phi(i) = k$. As $\pi(c)(i) =
c(k)$ and $\pi(d)(i) = d(k)$, it follows that $\pi(c)(i) \neq \pi(d)(i)$, so $\pi(c)
\neq \pi(d)$.
\end{proof}

The \definedterm{product metric} on $\functions{\Z}{2}$ is given by $\productmetric(c,d)
= \sum_{i \in \Z} \absolutevalue{c(i) - d(i)} / 2^{\absolutevalue{i}}$ for all
$c, d \in \functions{\Z}{2}$. For all $i \in I$, set $N_i =
\closedinterval{i}{i + w(i)}$ and $N_i' = \union[t < r(i)][N_{i,t}]$, where
$N_{i,t} = N_i + tq(i)$.

\begin{lemma}
$\pi$ is a homomorphism from
$\clique{n+1}{G} \union \injective{n}{\Cantorspace}$ to
$\distributionallyseparated{\bernoullishift}$.
\end{lemma}

\begin{proof}
Suppose that $c \in \functions{n'}{(\Cantorspace)}$ is in $\clique{n+1}{G} \union
\injective{n}{\Cantorspace}$. By permuting the coordinates of $c$, we can
assume that $c(0) = \constantsequence{0}{\N}$ if $\constantsequence{0}{\N}
\in \image{c}{n'}$. If $n' > n$, then fix $m \in \N$ such that $c(j) \mathrel{G_m}
c(k)$ for all $j < k < n'$. Suppose now that $\ell \in \N$ is sufficiently large that
$\projection{\functions{\ell}{2}} \composition c$ is injective, $\restriction{c(k)}
{\ell} \neq \constantsequence{0}{\ell}$ for all $0 < k < n'$, and $n' < 2^\ell$.
If $n' = n$, then let $p$ be the extension of $\inverse{(\projection{\functions
{\ell}{2}} \composition c)}$ in $P_\ell$, and fix $i \in I$ such that $\phi(i) = \pair
{\ell}{p}$. If $n' > n$, then let $p$ be the extension of $\inverse{(\projection
{\functions{\ell}{2}} \composition c)}$ in $P_{\ell,m,n'}$, and fix $i \in I$ such
that $\phi(i)=\quadruple{\ell}{m}{n'}{p}$. Note that $q(i)$ depends only on $n$
when $n' = n$, and only on the fixed values of $m$ and $n'$ otherwise, so it is
independent of $\ell$. If $t < r(i)$, then $\projection
{\functions{N_{i,t}}{2}} \composition \pi \composition c$ is injective, so
$\lowerproductmetric(S^{i'} \composition \pi \composition c) \ge 1/2^{q(i)}$
for all $i \le i'<i+q(i)r(i)$, thus
\begin{align*}
  \textstyle
  \cardinality{& \set{i'<i+q(i)r(i)}[\lowerproductmetric
    (S^{i'} \composition \pi \composition c)<1/2^{q(i)}]} /(i+q(i)r(i)) \\
    & \le i/(i+q(i)r(i)) \le 1/(1+i q(i)) \le 1/(1+i) \goesto 0
\end{align*}
as $i \goesto \infty$, and therefore as $\ell \goesto \infty$, hence $\pi \composition c \in \distributionallyseparated
{\bernoullishift}$.
\end{proof}

\begin{lemma}
  $\pi$ is a homomorphism from $\independent{n+1}{G}$ to $\setcomplement
  {\separated{\bernoullishift}}$.
\end{lemma}

\begin{proof}
  Suppose that $c \in \functions{n+1}{(\Cantorspace)}$ and $\pi \composition c$
  is separated. Fix $\epsilon > 0$ such that $I' = \set{i' \in \N}[\lowerproductmetric
  (\bernoullishift^{i'} \composition \pi \composition c) \ge \epsilon]$ is
  infinite, as well as $a \in \N$ such that $\sum_{\absolutevalue{i}>a} 1 /
  2^{\absolutevalue{i}} < \epsilon$. By throwing out finitely many elements of
  $I'$, we can assume that each $i' \in I'$ has the property that $i' > a$ and there is at most one $i \in I$ for
  which $\closedinterval{i'-a}{i'+a} \intersection N_i' \neq \emptyset$,
  so our choice of $a$ yields a unique such $i$, thus
  $\projection{\functions{N_{i,t}}{2}} \composition \pi \composition c$ is injective
  for all $t < r(i)$.
  As blocks of types (1) and (2) admit at most $n$ distinct patterns,
  there exists $n' > n$ for which $\phi(i)=\quadruple{\ell}{m}{n'}{p}$. As the restrictions
$\restriction{\pi(c(k))}{\closedinterval{i'-a}{i'+a}}$, for $k \le n$,
are distinct and have pairwise disjoint supports, the size of
  $\closedinterval{i'-a}{i'+a} \intersection \union[k \le n][\support(\pi(c(k)))]$ is at least
  $n$. As the difference between distinct points of this set is at least $m+1$,
  it follows that $(m+1)(n-1) \le 2a$.
  As at least two coordinates receive distinct positive labels,
  there exist $j<k<n+1$ with $c(j)\mathrel{G_{\ell,m}} c(k)$.
  By passing to an
  infinite subset of $I'$, we can assume that the same $j$, $k$, and $m$
  correspond to all elements of $I'$. Then there are infinitely many values of
  $\ell$ corresponding to the remaining elements of $I'$, since otherwise there
  are only finitely many possibilities for $n'$ and $p$, so
  only finitely many for $i$, thus only finitely many for $i'$, since $i' \in N_i' +
  \closedinterval{-a}{a}$. As $G_{\ell,m}$ is decreasing in $\ell$, it follows that
  $c(j) \mathrel{G_{\ell,m}} c(k)$ for all $\ell \in \N$, so $c(j) \mathrel{G_m} c(k)$,
  thus $c(j) \mathrel{G} c(k)$, hence $c \notin \independent{n+1}{G}$.
\end{proof}

Clearly, the closure $F$ of $\saturation{\image{\pi}{\Cantorspace}}{\bernoullishift}$ is
perfect.

\begin{lemma} \label{lemma:atmostcountable}
$F \setminus \saturation{\image{\pi}{\Cantorspace}}{\bernoullishift}$ is
countable.
\end{lemma}

\begin{proof}
As each $\image{\bernoullishift^i\composition\pi}{\Cantorspace}$
is compact, we need only consider limits of sequences of the form
$\sequence{\bernoullishift^{\pm i_k}(\pi(c_k))}[k \in \N]$, where $\sequence{c_k}[k \in \N]$ is a
sequence of elements of $\Cantorspace$
and $\sequence{i_k}[k \in \N]$ is a strictly increasing sequence of elements
of $\N$. Clearly $\bernoullishift^{-i_k}(\pi(c_k)) \goesto \constantsequence{0}
{\Z}$. Suppose that $\bernoullishift^{i_k}(\pi(c_k)) \goesto c$. If the distance
from $i_k$ to $\union[i \in I][N_i']$ goes to infinity, then $c =
\constantsequence{0}{\Z}$. Otherwise, by passing to a subsequence, we can assume
that there exist $i_k' \in I$ such that the distance from $i_k$ to $N_{i_k'}'$ is fixed.
If $w(i_k') \goesto \infty$, then every fixed neighborhood of $i_k$
eventually meets at most one $N_{i_k',t}$, and each $\pi(c_k)$ has
at most one non-zero coordinate in each such block, so
$\cardinality{\support(c)} \le 1$.
Otherwise, by passing to a subsequence, we can
assume that $w(i_k')$ is constant. On
$\closedopeninterval{i_k'}{i_k'+q(i_k')r(i_k')}$, each $\pi(c_k)$ is either
zero or the restriction of a periodic sequence of period $q(i_k')$.
Since $r(i_k') \goesto \infty$, the sequence $c$ is either periodic or
obtained from a periodic sequence by replacing all coordinates on one
side of an integer by zero.
\end{proof}

To see that each $c \in \functions{<\N}{\image{\pi}{\Cantorspace}}$ is
distributionally proximal, suppose that $\epsilon > 0$ and fix $a \in \N$ with
$\sum_{\absolutevalue{k}>a}1/2^{\absolutevalue{k}}<\epsilon$. Given
$i \in I$, let $i'$ be the next element of $I$. If $i + q(i)r(i) + a \le i'' < i' - a$,
then $\union[k<\length{c}][\support(c(k))] \subseteq \set{0} \union \union[i \in
I][N_i']$, so $\upperproductmetric(S^{i''} \composition c) < \epsilon$,
thus 
\begin{align*}
  \cardinality{\set{i'' < i'}[\upperproductmetric(S^{i''} \composition c) < \epsilon]} / i'
    & \ge 1 - (i + q(i)r(i) + 2a) / i'  \\
    & \ge 1 - (1 / i) - (2a / i') \goesto 1
\end{align*}
as $i \goesto \infty$.
\end{proof}

\begin{proposition}[$\ZF + \DC$] \label{proposition:perfect}
Suppose that $n \ge 3$, $X$ and $Y$ are \Polish spaces, $G$ is a \Borel
graph on $X$, $T \from Y \to Y$ is a \Borel automorphism, $B
\subseteq Y$ is an $n$-scrambled uncountable \Borel set,
$\pi \from X \to Y$ is an injective \Borel
homomorphism from $\independent{n}{G}$ to $\setcomplement{\LiYorke{T}}$, and
$\saturation{\image{\pi}{X}}{T}$ is co-countable. Then $G$ has
a perfect clique.
\end{proposition}

\begin{proof}
Fix $i \in \Z$ for which $B \intersection \image{(T^i \composition \pi)}
{X}$ is uncountable. The perfect set theorem and \Galvin's Theorem (see, for example, \cite
[Theorems 13.6 and 19.7]{Kechris}) then yield a perfect set $P \subseteq \preimage
{(T^i \composition \pi)}{B}$ that is either a $G$-clique or $G$-independent.
But $G$-independence would contradict the fact that $B$ is $n$-scrambled.
\end{proof}

\section{$n$-Scrambled sets} \label{section:consequences}

We begin this section with the proof of our primary result:

\begin{proof}[Proof (of Theorem \ref{introtheorem:main})]
By \cite[Theorem 2.1]{KubisVejnar} and a simple \Baire category
argument, there is an $\Fsigma$ graph $G$ on $\Cantorspace$ with an
uncountable clique $Y$ but no perfect clique.
By Theorem \ref{introtheorem:technical}, there exist a homeomorphism
$T \from \Cantorspace \to \Cantorspace$ and a continuous injective
homomorphism $\pi \from \Cantorspace \to \Cantorspace$ from
$\pair{\injective{n}{\Cantorspace} \union \clique{n+1}{G}}{\independent{n+1}{G}}$ to
$\pair{\distributionallyscrambled{T}}{\setcomplement{\LiYorke{T}}}$ for which
$\saturation{\image{\pi}{\Cantorspace}}{T}$ is co-countable. Then
$\image{\pi}{\Cantorspace}$ is distributionally $n$-scrambled,
$\image{\pi}{Y}$ is a distributionally finitely scrambled uncountable set, and
Proposition \ref{proposition:perfect} rules out an $(n+1)$-scrambled \Cantor set. 
\end{proof}

Let $\continuum$ denote the cardinality of the continuum.

\begin{theorem}
It is consistent with $\ZFC + \neg \CH$ that, for all $n \ge 2$, there is a homeomorphism
$T \from \Cantorspace \to \Cantorspace$ that has a distributionally $n$-scrambled
\Cantor set containing a distributionally finitely scrambled set of cardinality
$\continuum$, but does not have an $(n+1)$-scrambled \Cantor set.
\end{theorem}

\begin{proof}
Applying \cite[Theorem 1.13]{Shelah} with
$\lambda=\mu=\aleph_2$ over a model of $\mathtt{GCH}$, we can assume
that $\continuum=\aleph_2$ and that there is an $\Fsigma$ graph $G$
on $\Cantorspace$ with a clique of cardinality $\continuum$ but
no perfect clique. The proof of Theorem \ref{introtheorem:main}
therefore yields the desired result.
\end{proof}

A \definedterm{hypergraph} is a permutation-invariant set of
injective sequences.

\begin{theorem} \label{theorem:consistency}
It is consistent with $\ZFC + \neg \CH$ that, for all $n \ge 2$, every \Borel
function on a \Polish space with a (distributionally) $n$-scrambled set of cardinality
$\aleph_2$ has a (distributionally) $n$-scrambled perfect set.
\end{theorem}

\begin{proof}
By \cite[Proposition 3.4]{KubisShelah}, we can assume that every
\Borel $n$-ary hypergraph on a \Polish space with a clique of
cardinality $\aleph_2$ has a perfect clique. But $n$-ary
(distributional) scrambledness is \Borel.
\end{proof}

\begin{remark}
The forcing argument underlying
\cite[Proposition 3.4]{KubisShelah} also yields the analog of
Theorem \ref{theorem:consistency} for (distributionally) finitely
scrambled sets. To see this, note that---after the usual
$\Delta$-system reduction to a common name---the $k$-fold
\Cohen product forces that the corresponding evaluations form a
(distributionally) $k$-scrambled $k$-element set for all $k \ge 2$,
and the perfect-set construction can meet the requirements for all finite
arities simultaneously.
\end{remark}

\begin{acknowledgements}
I would like to thank Alexander Kechris for encouraging me to utilize
ChatGPT-5.6 Sol, which provided substantial assistance in the discovery
and development of the results.
\end{acknowledgements}

\bibliographystyle{amsalpha}
\bibliography{bibliography}

\end{document}